\documentclass{amsart}[12pt]

\usepackage{amsmath,amssymb,enumerate}

\usepackage{color}
\usepackage{mathrsfs}

\theoremstyle{plain}
\newtheorem{thm}{Theorem}[section]
\newtheorem{theorem}[thm]{Theorem}
\newtheorem{lemma}[thm]{Lemma}
\newtheorem{lem}[thm]{Lemma}
\theoremstyle{definition}
\newtheorem{defi}[thm]{Definition}

\newtheorem{prop}[thm]{Proposition}
\newtheorem{rem}[thm]{Remark}
\newtoks\by
\newtoks\paper
\newtoks\book
\newtoks\jour
\newtoks\yr
\newtoks\pages
\newtoks\vol
\newtoks\publ
\def\ota{{\hbox\vol{???}}}
\def\cLear{\by=\ota\paper=\ota\book=\ota\jour=\ota\yr=\ota
	\pages=\ota\vol=\ota\publ=\ota}
\def\endpaper{\the\by, \the\paper.
	{\it\the\jour\/} {\bf \the\vol} (\the\yr), \the\pages.\cLear}
\def\endbook{\the\by, {\it\the\book}. \the\publ.\cLear}
\def\endprep{\the\by, \the\paper. \the\jour.\cLear}

\def\spa{\textup{span}}

\def\FT{\mathscr F}

\def\Ne{{\mathbb{N}}}

\newcommand{\eps}{\varepsilon}
\def\N{\mathbb{N}} 
\def\Z{\mathbb{Z}}
\def\R{\mathbb{R}}
\def\be{\begin{equation}}
	\def\ee{\end{equation}}

\def\sT{\,;\;}
\def\dd{\text{\rm\,d}} 

\newcommand{\Bs}{\mathscr{B}}

\newcommand{\Fs}{\mathscr{F}}

\newcommand{\Ss}{\mathscr{S}}

\newcommand{\T}{\mathbb{T}}

\numberwithin{equation}{section}

\newtheorem{cor}[thm]{Corollary}

\theoremstyle{definition}

\newtheorem{defn}[thm]{Definition}

\newtheorem*{defn*}{Definition}

\newtheorem{rmk}[thm]{Remark}

\newtheorem*{rmk*}{Remark}

\numberwithin{equation}{section} \headheight=12pt
\begin{document}
	\title[Notes on Non-Compact Maps and Bernstein Numbers]{Notes on Non-Compact Maps and \\  the Importance of Bernstein Numbers}

	\author[D. E. Edmunds]{David E. Edmunds}
	\address{Department of Mathematics\\ University of Sussex\\ Pevensey 2 Building\\ Brighton\\
		BN1 9QH, Sussex\\ United Kingdom \\ {https://orcid.org/0000-0002-8528-3067}}
	\email{davideedmunds@aol.com}

		\author[J. Lang]{Jan Lang}
	\address{Department of Mathematics, The Ohio State University\\ Columbus, OH \\ United States \  \\ {\it and }\\	
		Department of Mathematics, Faculty of Electrical Engineering\\ Czech Technical University in Prague \\ Czech Republic \\
		{https://orcid.org/0000-0003-1582-7273}}
	\email{lang@math.osu.edu}
	
		\date{}
	\thanks{This paper is dedicated to the memory of Prof. A. Pietsch and his contributions to mathematics.}

	\begin{abstract}
		In this review paper we study non-compact operators and embeddings between function spaces, highlighting interesting phenomena and the significance of Bernstein numbers. In particular, we demonstrate that for non-compact maps the usual $s$-numbers (e.g., approximation, Kolmogorov, and entropy numbers) fail to reveal finer structural properties, and one must instead consider concepts such as strict singularity and Bernstein numbers.
	\end{abstract}
	\maketitle
	
	\section{Introduction}
	
	Let $T: X \to Y$ be a linear map between Banach spaces $X$ and $Y$. When $T$ is compact, one may use entropy numbers or $s$-numbers (such as approximation and Kolmogorov numbers) to describe the quality of compactness and the speed of approximation via finite-dimensional subspaces. An extensive literature exists on compact operators and their approximative characteristics. 
	
	However, in this paper we focus on non-compact maps. We show that when the embeddings or maps are non-compact, the aforementioned numbers do not provide useful information about the inner structure of the map. Instead, one needs to employ alternative quantities such as strict singularity and Bernstein numbers. 
	
	We recall that, except for the entropy numbers, the numbers mentioned above belong to the class of $s$-numbers, which were extensively studied by Pietsch throughout his career. It is worth noting that among the $s$-numbers, Bernstein numbers—introduced in Pietsch's seminal paper from 1970—were later largely neglected. In fact, in his later survey ``History of Banach Spaces'' he did not mention them, and in a subsequent paper titled ``Bad Properties of the Bernstein Numbers,'' he described certain drawbacks of the operator ideals based on these numbers.

Our aim is to restore Bernstein numbers to their former prominence and to demonstrate that they offer additional insights into the structure of non‐compact maps.  In particular, we focus on the following non‐compact maps and embeddings:
\begin{enumerate}[(a)]
	\item Embeddings between Besov spaces,
	\item Embeddings between Sobolev spaces,
	\item The Fourier transformation,
	\item The Laplace transformation.
\end{enumerate}

Our analysis employs the concepts of strictly singular and finitely strictly singular operators, together with the behavior of Bernstein numbers, to describe the \emph{quality} of non‐compactness.

In the next section, we provide further motivation for studying Bernstein numbers and introduce the essential definitions.  We show that Bernstein numbers yield additional information about non‐compact embeddings, thereby justifying their study.  The subsequent sections present detailed analyses of the maps listed above; each was published separately in recent years (see \cite{ChianJanLidingI} for (a), \cite{ChianJanLidingII} for (b), \cite{EdmundsGurkaLang} for (c), and \cite{EdmundsLangLaplace} for (d)).

	\section{Motivation and Main Results} \label{section2}
	
	In this section we demonstrate—via concrete examples—why the standard tools used for describing (non-)compactness are not sufficiently refined and why strict singularity and finite strict singularity can provide additional insight.
	
	Let $T: X \to Y$ be a linear map between Banach spaces. We denote such maps by $T \in B(X,Y)$, and for $\varepsilon > 0$ we define
	\[
	B_\varepsilon X := \{ x \in X : \|x\|_X \le \varepsilon \},
	\]
	with the notation $BX := B_1X$.
	
	One fundamental question in functional analysis is to identify a quantity that measures the degree of compactness and, more generally, the size of the non-compact part of $T$. For compact operators, the most natural quantities are the $s$-numbers (such as approximation numbers and Kolmogorov numbers) as well as entropy numbers. We recall their definitions.
	
	\begin{defi}
		Let $T\in B(X,Y)$. Then:
		\begin{itemize}
			\item[(i)] The $n$th \emph{approximation number} of $T$ is defined by
			\[
			a_n(T) := \inf \{ \|T-F\| : F\in B(X,Y),\, \operatorname{rank}(F) < n \}.
			\]
			\item[(ii)] The $n$th \emph{Kolmogorov number} of $T$ is defined by
			\[
			d_n(T) := \inf_{Y_n} \sup_{z\in T(BX)} \inf_{y \in Y_n} \|y-z\|_Y,
			\]
			where the infimum is taken over all $n$-dimensional subspaces $Y_n \subset Y$.
			\item[(iii)] The $n$th \emph{entropy number} of $T$ is defined by
			\[
			e_n(T) := \inf \left\{\varepsilon > 0 : T(BX) \subset \bigcup_{j=1}^{2^{n-1}} \left(b_j + \varepsilon BY\right) \text{ for some } b_1, \dots, b_{2^{n-1}} \in Y \right\}.
			\]
		\end{itemize}
	\end{defi}
	
	It is well known that
	\[
	\|T\| = a_1(T) \ge a_2(T) \ge \cdots,
	\]
	and similarly for the Kolmogorov and entropy numbers. Moreover, if we restrict our attention to spaces with the approximation property, then $T\in B(X,Y)$ is compact if and only if
	\[
	\lim_{n\to\infty} a_n(T) = \lim_{n\to\infty} d_n(T) = \lim_{n\to\infty} e_n(T) = 0.
	\]
	
	We recall that the \emph{ball measure of non-compactness} $\alpha(T)$ of a bounded linear map $T: X \to Y$ is defined as the infimum of all radii $\rho > 0$ for which there exists a finite collection of balls in $Y$, each of radius $\rho$, whose union covers $T(B_X)$ (here $B_X$ denotes the unit ball of $X$). Clearly,
	\[
	0 \le \alpha(T) \le \|T\|.
	\]
	In our context we have
	\[
	\alpha(T)= \lim_{n\to \infty} a_n(T)= \lim_{n\to \infty} d_n(T)=\lim_{n\to \infty} e_n(T).
	\]
	
	We say that the map $T$ is \emph{maximally non-compact} if $\alpha(T) = \|T\|$.
	
	One might expect that the ball measure of non-compactness should fully describe the quality of non-compactness. However, as the following examples show, it is not sufficiently sensitive.
	
	Consider the following standard non-compact embeddings:
	\begin{itemize}
		\item[I.] The embedding $I_1:\ell^{p_0}(\N) \hookrightarrow\ell^{p_1}(\N)$, for $p_0 \le p_1$.
		\item[II.] The embedding $I_2:\ell^{p,q_0}(\N) \hookrightarrow\ell^{p,q_1}(\N)$, for $q_0 \le q_1$.
	\end{itemize}
	
	It is well known that these maps are non-compact, and in fact one easily verifies that
	\begin{align*}
		\|I_1\|_{\ell^{p_0}\to\ell^{p_1}} &= 1 = \alpha(I_1), \quad \text{for } p_0 \le p_1, \\
		\|I_2\|_{\ell^{p,q_0}\to\ell^{p,q_1}} &= 1 = \alpha(I_2), \quad \text{for } q_0 \le q_1.
	\end{align*}
	Thus, both $I_1$ and $I_2$ are maximally non-compact. Moreover, the embeddings \( \ell^{p_0}\hookrightarrow \ell^{p_1} \) and \( \ell^{p,q_0}\hookrightarrow \ell^{p,q_1} \) are strict when $p_0 < p_1$ and $q_0 < q_1$, respectively. Since the ball measure of non-compactness remains constant even as the target space enlarges, it fails to detect qualitative changes in the non-compactness of $I_1$ and $I_2$.
	
	Next, consider the Sobolev embedding:
	\begin{itemize}
		\item[III.] The embedding 
		\[
		E: W^{1,p}(\Omega) \to L^{p^*,q}(\Omega),
		\]
		where $\Omega\subseteq\R^d$, $1<p<d$, $1<p\le q\le\infty$, and $p^*=\frac{dp}{d-p}$.
	\end{itemize}
	
	Note that $L^{p^*,p}$ is the smallest possible target space for the Sobolev space $W^{1,p}$, not only among Lorentz spaces but also among all rearrangement-invariant spaces. In this context, one finds (see \cite{LMOP} and the references therein) that
	\[
	\|E\| = \alpha(E),
	\]
	so the embeddings into the Lorentz spaces $L^{p^*,q}$ (with $p\le q\le\infty$) are all maximally non-compact. As in the previous examples, since $L^{p^*,q} \hookrightarrow L^{p^*,r}$ for $q < r$, the ball measure of non-compactness does not capture the qualitative change in non-compactness as the target space is enlarged.
	
	This motivates the search for alternative quantities to measure the \emph{degree} of non-compactness. A more refined notion arises from the concepts of \emph{strict singularity}, \emph{finite strict singularity}, and \emph{Bernstein numbers}.
	
	Given a bounded linear map $T: X \to Y$ between two (quasi-)Banach spaces, the $n$th \emph{Bernstein number} $b_n(T)$ is defined as
	\[
	b_n(T) := \sup \left\{ \inf_{\substack{x \in X_n \\ \|x\|_X=1}} \|Tx\|_{Y} : X_n \subseteq X \text{ is an } n\text{-dimensional subspace} \right\}.
	\]
	One can show (see \cite[Theorem~4.6]{P:74}) that the Bernstein numbers are the smallest injective strict $s$-numbers. We have
	\[
	\|T\| = b_1(T) \ge b_2(T) \ge \cdots,
	\]
	and in general, 
	\[
	a_n(T) \ge d_n(T) \ge b_n(T).
	\]
	
	We say that $T$ is \emph{finitely strictly singular} if
	\[
	\lim_{n\to\infty} b_n(T)=0.
	\]
	The operator $T$ is \emph{strictly singular} if there is no infinite-dimensional subspace $Z\subseteq X$ such that the restriction $T|_Z$ is an isomorphism from $Z$ onto $T(Z)$. Equivalently,
	\[
	\inf_{\substack{x \in Z \\ \|x\|_X=1}} \|Tx\|_{Y} = 0 \quad \text{for every infinite-dimensional subspace } Z \subseteq X.
	\]
	
	These notions are related as follows:
	\[
	\text{$T$ is compact} \quad \Longrightarrow \quad \text{$T$ is finitely strictly singular} \quad \Longrightarrow \quad \text{$T$ is strictly singular},
	\]
	and in general the reverse implications fail. It is also known that if $T$ is an operator between Hilbert spaces then $T$ is strictly singular if and only if it is compact; in this case, all $s$-numbers coincide.
	
	The importance of strict singularity for non-compact maps can be seen in the following observations:
	\begin{itemize}
		\item For the embedding $I_1:\ell^{p_0} \hookrightarrow \ell^{p_1}$ with $p_0 < p_1$, one obtains
		\[
		b_n(I_1) = n^{\frac{1}{p_1} - \frac{1}{p_0}},
		\]
		which implies that $I_1$ is finitely strictly singular; when $p_0=p_1$, it is not strictly singular.
		\item For the embedding $I_2:\ell^{p,q_0} \hookrightarrow \ell^{p,q_1}$ with $q_0 < q_1$, one finds that $I_2$ is strictly singular but not finitely strictly singular, since $b_n(I_2)=1$.
		\item For the non-compact Sobolev embedding, Bourgain and Gromov \cite{BouGro} showed that
		\[
		W^{1,1} ((0,1)^d) \to L^{\frac{d}{d-1}}((0,1)^d)
		\]
		has Bernstein numbers that decay to zero, so it is finitely strictly singular. Later, Lang and Mihula \cite{LangMihula} improved and extended these results, showing that the optimal Sobolev embedding
		\[
		W^{k,p} (\Omega) \to L^{p^*,p}(\Omega)
		\]
		is not strictly singular, whereas the embedding
		\[
		W^{k,p} (\Omega) \to L^{p^*}(\Omega)
		\]
		is finitely strictly singular, with sharp estimates for the decay of Bernstein numbers.
	\end{itemize}
	
	Thus, strict singularity and finite strict singularity provide additional information on the structure of non-compact embeddings between function spaces.
	
	A natural question is whether strict singularity can yield new insights beyond the framework of embeddings between function spaces, particularly for general linear operators. For example, consider the discrete Fourier transform
	\[
	\mathcal{F}: L^p({\mathbb{S}}^1) \to \ell^{p'}(\Z), \quad 1 \leq p \leq 2.
	\]
	It is clear that $\mathcal{F}$ is non-compact. Results of Weis (1987) and Lefèvre \& Rodríguez-Piazza (2014) (see \cite{LefevrePiazzaFSSApplication}) imply that:
	\begin{itemize}
		\item If $1 < p < 2$, then $\mathcal{F}$ is finitely strictly singular.
		\item If $p = 1$, then $\mathcal{F}$ is strictly singular but not finitely strictly singular.
		\item If $p = 2$, then $\mathcal{F}: L^2 \to \ell^2$ is not strictly singular (in fact, it is invertible).
	\end{itemize}
	This demonstrates that the notion of strict singularity is valuable not only for non-compact embeddings between function spaces but also for more general non-compact maps.
	The results mentioned in the next sections are concerned with:
	\begin{itemize}
		\item Embeddings between Besov sequence spaces (see Corollary \ref{Cor::SeqEmbed2}),
		\item Embeddings between Besov spaces (see Theorems \ref{Thm::ClassifyDom} and \ref{Thm::ClassifyRn}),
		\item Embeddings between Sobolev spaces (see Theorem \ref{Thm::MainThm}),
		\item The Fourier transformation (see Theorem \ref{MainThhm}),
		\item The Laplace transformation (see Theorems \ref{Thereom 1} and \ref{Thereom 2}).
	\end{itemize}
	All these results have 
	 appeared in earlier publications; here we assemble them to provide an overview of recent results in non‐compact maps and embeddings, and their connection with Bernstein numbers. For further details (including notation), see \cite{ChianJanLidingI}, \cite{ChianJanLidingII}, \cite{EdmundsGurkaLang} and \cite{EdmundsLangLaplace}.

	\section{Definitions of Function Spaces and Technical Results} 
	
	Natural questions following from the work of Lang and Mihula \cite{LangMihula} involve describing the Sobolev embedding when the target space is $L^{p^*,q}$ with $p < q \neq p^*$. This case was omitted in their work since their techniques did not extend to this range. We will obtain information about the Sobolev embedding by studying embeddings between Besov spaces.
	
	We start with some definitions and notations. First, we recall the definitions of Lorentz, Sobolev, and Besov spaces.
	
	Given a measurable function $f:\R^n\to\R$, its \emph{distribution function} $f_*$ is defined by
	\begin{equation*} 
		f_{*}(\tau)=\mu\{x\in\R^n:|f(x)|>\tau\},\quad \tau>0,
	\end{equation*}
	and its \emph{nonincreasing rearrangement} is given by
	\begin{equation} \label{NonRear}
		f^*(t)=\inf\{\tau>0: f_{*}(\tau)\le t\},\quad t>0.
	\end{equation}
	
	\begin{defn}\label{Defn::Lorentz}
		Let $(\Omega,\mu)$ be a measure space. For $0<p<\infty$ and $0<r\le\infty$, the \emph{Lorentz space} $L^{p,r}(\mu)$ consists of all measurable functions $f:\Omega\to\R$ such that
		\[
		\|f\|_{L^{p,r}(\mu)}=p^{\frac1r}\left(\int_0^\infty t^{r-1}\mu\{x:|f(x)|>t\}^{r/p}\,dt\right)^{1/r}<\infty.
		\]
		When $r=\infty$, the integral is replaced by the appropriate supremum. In particular, $L^{p,p}(\mu)=L^p(\mu)$, and if $0<r_0<r_1\le\infty$, then $L^{p,r_0}(\mu)\subset L^{p,r_1}(\mu)$.
	\end{defn}
	
	For convenience, we set $L^{\infty,\infty}(\mu)=L^\infty(\mu)$ and note that $L^{\infty,r}(\mu)=\{0\}$ for $r<\infty$.
	
	We also require some basic properties of the distribution function and the nonincreasing rearrangement (cf. \cite[Proposition~3.1]{LangMihula}):
	
	\begin{lem}\label{SumDistrF}
		Let $\{M_j\}_{j\in\N}$ be a sequence of pairwise disjoint measurable subsets of $\R^n$. If $\{\psi_j\}_{j=1}^\infty$ is a sequence of measurable real-valued  functions with $\operatorname{supp}\psi_j\subset M_j$, then for each $\tau>0$,
		\[
		\Big(\sum_{j=1}^{\infty}\psi_j\Big)_{*}(\tau)
		=
		\sum_{j=1}^{\infty} (\psi_j)_{*}(\tau).
		\]
	\end{lem}
	
	\begin{lem}\label{SumNonIncrRear} 
		Under the same assumptions as in Lemma \ref{SumDistrF} and if in addition $\{I_j\}_{j\in\N}$ is a sequence of pairwise disjoint intervals in $(0,\infty)$, then for every $t>0$,
		\[
		\Big(\sum_{j=1}^{\infty}\psi_j\Big)^{*}(t)
		\ge
		\sum_{j=1}^{\infty}\chi_{I_j}(t)(\psi_j)^{*}(t).
		\]
	\end{lem}
	
	\begin{proof}
		The inequality is trivial for $t\notin\bigcup_{j\in\N}I_j$. For $t\in I_\ell$ (with $\ell$ unique),
        by \eqref{NonRear} there is a number
	$\tau_0>0$ such that
	\begin{equation*}
		\mu\big\{x\in\R^d\sT\big|{\textstyle\sum_{j=1}^{\infty}\psi_j(x)}\big|>\tau_0\}\le t
		\quad
		\text{and}
		\quad
		\tau_0<(\psi_\ell)^{*}(t).
	\end{equation*}
	On the other hand, \eqref{NonRear} and Lemma \ref{SumDistrF}   imply that
	\begin{equation*}
		\mu\{x\in\R^d\sT |\psi_\ell(x)|>\tau_0\}>t
	\end{equation*}
	and we get a contradiction, and the assertion is proved.
	\end{proof}
	
	\begin{defn}\label{Defn::Sobolev}
		Let $\Omega\subseteq\R^n$ be an open set. For $m\ge0$, define the Sobolev-Lorentz space $W^{m}L^{p,r}(\Omega)$ (when $1<p\le\infty$ or $p=1\ge r$) as the set of functions in $L^1(\Omega)$ whose derivatives up to order $m$ (in the distributional sense) belong to $L^{p,r}(\Omega)$, with norm
		\begin{equation*}
			\|f\|_{W^{m}L^{p,r}(\Omega)} := \sum_{|\alpha|\le m} \|\partial^\alpha f\|_{L^{p,r}(\Omega)} <\infty.
		\end{equation*}
		We denote by $W_0^mL^{p,r}(\Omega)$ the closure of $C_c^\infty(\Omega)$ under this norm.
	\end{defn}

Note that the above spaces are labeled as $W^{m,(p,r)}(\Omega)$ in some literature.
For the case $p=1$, we require $r\le 1$ in order to have  $L^{1,r}\subset L^1 $ so that the elements in Sobolev-Lorentz spaces make sense as distributions.
	
	\begin{defn}\label{Defn::BesovRn}
		Let $s\in\R$ and $p,q\in(0,\infty]$. The Besov space $\Bs_{pq}^s(\R^n)$ consists of all tempered distributions $f\in\Ss'(\R^n)$ for which
		\[
		\|f\|_{\Bs_{pq}^s(\R^n)} = \Big( \sum_{j=0}^{\infty} \| 2^{js} \varphi_j * f\|_{L^p(\R^n)}^q \Big)^{\frac{1}{q}} <\infty.
		\]
		Here $(\varphi_j)_{j=0}^\infty\subset\Ss(\R^n)$ is a fixed family of Schwartz functions whose Fourier transforms satisfy: $\operatorname{supp}\widehat{\varphi}_0\subset B(0,2)$ with $\widehat{\varphi}_0\equiv1$ on $\overline{B(0,1)}$, and for $j\ge1$, $\widehat{\varphi}_j(\xi)=\widehat{\varphi}_0(2^{-j}\xi)-\widehat{\varphi}_0(2^{1-j}\xi)$. Different choices of $(\varphi_j)_j$ yield equivalent norms (see, e.g., \cite[Proposition~2.3.2]{TriebelTheoryOfFunctionSpacesI}).
	\end{defn}
	
	\begin{defn} (See also
		 \cite[Definition~1.95]{TriebelTheoryOfFunctionSpacesIII})
		 \label{Defn::BesovDom}
		For an open set $\Omega\subseteq\R^n$, define
		\[
		\Bs_{pq}^s(\Omega) := \{ \tilde f|_\Omega : \tilde f \in \Bs_{pq}^s(\R^n) \}, \quad \|f\|_{\Bs_{pq}^s(\Omega)} := \inf_{\tilde f|_\Omega=f}\|\tilde f\|_{\Bs_{pq}^s(\R^n)}.
		\]
		Moreover, define the closed subspace
		\[
		\widetilde \Bs_{pq}^s(\overline\Omega) := \{\tilde f\in \Bs_{pq}^s(\R^n):\operatorname{supp}\tilde f\subset\overline\Omega\}.
		\]
	\end{defn}
	
	Note that when $\Omega=\R^n$, one has
	\[
	\Bs_{pq}^s(\R^n) = \widetilde \Bs_{pq}^s(\overline{\R^n}).
	\]
	
	An important property of Besov spaces is that they can be characterized via sequence spaces. For this purpose, we introduce the following:
	
	\begin{defn}\label{Defn::VectEllSpace}
		Let $(X_j)_{j=0}^\infty$ be a family of quasi-normed spaces and let $q\in(0,\infty]$. The vector-valued sequence space $\ell^q(X_j)_{j=0}^\infty$ (sometimes denoted by $\ell^q(\N_0;(X_j)_{j=0}^\infty)$) is defined as the set of sequences $x=(x_0,x_1,x_2,\dots)$ with $x_j\in X_j$ for each $j\ge0$ such that
		\[
		\|x\|_{\ell^q(X_j)_{j=0}^\infty} := \left(\sum_{j=0}^\infty \|x_j\|_{X_j}^q\right)^{1/q} <\infty.
		\]
            We will use the notation  $x=\sum_{j=0}^\infty x_j\otimes e_j$ where $e_j=(0,\dots,0,\underset j1,0,\dots)$ is the standard basis in $\ell^q$.

		When $q=\infty$, the sum is replaced by the supremum. If all $X_j$ are identical (say $X_j\equiv Y$), we write $\ell^q(Y)$.
	\end{defn}
	
	\begin{rmk}\label{Rmk::VectEllSpace}
		In the notation for the double sequence space $\ell^q(\ell^p)$, by the element $e_j\otimes e_k$ we mean the tensor product of the standard basis vectors in $\ell^p$ and $\ell^q$, respectively.
	\end{rmk}
	
	\section{Embeddings between Besov Spaces and Sobolev Embeddings}

	The results in this section are based on the papers \cite{ChianJanLidingI} and \cite{ChianJanLidingII}, whose notation we follow here.

	\subsection{Besov Sequence Spaces}
	
	The following proposition (due to Ciesielski and Figiel; see \cite{CF1,CF2}) provides a key link between Besov spaces and sequence spaces. For this purpose, define for $j\ge0$ the index sets
	\[
	Q_j:=2^{-j}\Z^n \quad \text{and} \quad G_0:=\{0,1\}^n,\quad G_j:=\{0,1\}^n\setminus\{0\}^n \text{ for } j\ge1.
	\]
	
	\begin{prop}[{\cite[Theorem~1.20]{TriebelSpacesOnDomains},\cite[Prop. 2.12]{ChianJanLidingI}}] \label{Prop::Wavelet}
		Let $\eps>0$. There exists a set of functions
		\[
		\{\psi_{jm}^g: j\ge0,\; m\in Q_j,\; g\in G_j\}\subset C_c^\infty(\R^n)
		\]
		(depending on $\eps$) that forms a (real) orthonormal basis for $L^2(\R^n)$ and satisfies:
		\begin{enumerate}[(i)]
			\item \label{Item::Wavelet::Scal} For all $j\ge0$, $m\in Q_j$, and $g\in G_j$, 
			\[
			\psi_{jm}^g(x)=2^{jn/2}\psi_{00}^g(2^j(x-m)).
			\]
			\item \label{Item::Wavelet::Supp} There exists an $L\ge1$ such that $\operatorname{supp}\psi_{00}^g\subset B(0,2^L)$ for all $g\in G_0$, and hence $\operatorname{supp}\psi_{jm}^g\subset B(m,2^{L-j})$.
			\item \label{Item::Wavelet::Sum} For every $p,q\in[\eps,\infty]$, $s\in[-\eps^{-1},\eps^{-1}]$, and $f\in \Bs_{pq}^s(\R^n)$, there is an expansion (which is orthonormal in $L^2$)
			\[
			f=\sum_{j=0}^\infty\sum_{\substack{m\in Q_j \\ g\in G_j}} (f,\psi_{jm}^g)\,\psi_{jm}^g,
			\]
			where the series converges in the sense of tempered distributions.
			\item \label{Item::Wavelet::Isom} For every $p,q\in[\eps,\infty]$ and $s\in[-\eps^{-1},\eps^{-1}]$, 
			\begin{equation}\label{Eqn::Wavelet::bpqsSpace}
				f\in \Bs_{pq}^s(\R^n) \quad \Longleftrightarrow \quad \left(\sum_{j=0}^\infty 2^{j(s-\frac{n}{p}+\frac{n}{2})q}\left(\sum_{(m,g)\in Q_j\times G_j}|(f,\psi_{jm}^g)|^p\right)^{\frac{q}{p}}\right)^{\frac{1}{q}} <\infty.
			\end{equation}
			In other words, one obtains the isomorphism
			\begin{equation}\label{Eqn::Wavelet::Lambda}
				\Lambda:\Bs_{pq}^s(\R^n) \xrightarrow{\simeq} \ell^q\Big(2^{j(s-\frac{n}{p}+\frac{n}{2})}\cdot\ell^p(Q_j\times G_j)\Big)_{j=0}^\infty,
			\end{equation}
               
            \[\sum_{j=0}^\infty\sum_{m,g}a_{jm}^g\psi_{jm}^g 
            \longleftarrow
 \{a_{jm}^g:m\in Q_j,\ g\in G_j\}_{j=0}^\infty.\]
			where
			\[
			\Lambda(f)=\{(f,\psi_{jm}^g)\}_{j,m,g}.
			\]
			By fixing identifications $Q_j\times G_j\simeq\Z$, one may write
			\[
			\Lambda:\Bs_{pq}^s(\R^n) \xrightarrow{\simeq} \ell^q\Big(2^{j(s-\frac{n}{p}+\frac{n}{2})}\cdot\ell^p(\Z)\Big)_{j=0}^\infty.
			\]
         
		\end{enumerate}
	\end{prop}
	
	\begin{rmk}
		The right-hand side of \eqref{Eqn::Wavelet::bpqsSpace} is sometimes referred to as the Besov sequence norm (see \cite[Definition~1.18]{TriebelSpacesOnDomains}).
	\end{rmk}
	
	\begin{rmk}
		For Besov spaces on bounded Lipschitz domains $\Omega$, there exists a linear map $\tilde\Lambda$ such that
		\[
		\tilde\Lambda:\Bs_{pq}^s(\Omega) \xrightarrow{\simeq} \ell^q\Big(2^{j(s-\frac{n}{p}+\frac{n}{2})}\cdot\ell^p\{1,\dots,2^{nj}\}\Big)_{j=0}^\infty
		\]
		 is an isomorphism for all $p,q\in[\eps,\infty]$ and $s\in[-\eps^{-1},\eps^{-1}]$.
	\end{rmk}
	
	To study embeddings between Besov spaces, we consider the following embeddings between sequence spaces:
	\begin{itemize}
		\item[(i)]  $\ell^{q_0}(\ell^{p_0})\hookrightarrow\ell^{q_1}(\ell^{p_1})$,
		\item[(ii)]  $\ell^{q_0}(2^{js}\cdot\ell^{p_0})_{j=0}^\infty\hookrightarrow\ell^{q_1}(2^{jt}\cdot\ell^{p_1})_{j=0}^\infty$,
		\item[(iii)] $ \ell^{q_0}(2^{-j/{p_0}}\cdot\ell^{p_0}\{1,\dots,2^j\})_{j=A}^\infty\hookrightarrow\ell^{q_1}(2^{-j/{p_1}}\cdot\ell^{p_1}\{1,\dots,2^j\})_{j=A}^\infty$,
		\item[(iv)] $\ell^{q_0}(L^\infty[0,1])\hookrightarrow\ell^{q_1}(L^{p_1}[0,1])$.
	\end{itemize}
	
	We now state a few key results concerning embeddings between these sequence spaces.  We refer to    \cite{ChianJanLidingI} for a more complete list of embeddings between  Besov Sequence spaces. 
	
	\begin{prop}(\cite[Proposition 3.5]{ChianJanLidingI}).\label{Prop::SeqEmbed1}
		Let $0<p_0<p_1\le\infty$ and $0<q_0<q_1\le\infty$. Then
		\[
		b_n\Big(\ell^{q_0}(\ell^{p_0})\hookrightarrow\ell^{q_1}(\ell^{p_1})\Big) \le n^{-\min\big(\frac{1}{p_0},\frac{1}{q_0}\big)\Big(1-\max\big(\frac{q_0}{q_1},\frac{p_0}{p_1}\big)\Big)}.
		\]
		In particular, the embedding $\ell^{q_0}(\ell^{p_0})\hookrightarrow\ell^{q_1}(\ell^{p_1})$ is finitely strictly singular.
	\end{prop}
	
	\begin{proof}
		Using Hölder’s inequality, one obtains
		\[
		\|x\|_{\ell^{q_1}(\ell^{p_1})} \le \|x\|_{\ell^{q_0}(\ell^{p_0})}^{\max\big(\frac{q_0}{q_1},\frac{p_0}{p_1}\big)} \|x\|_{c_0}^{1-\max\big(\frac{q_0}{q_1},\frac{p_0}{p_1}\big)}.
		\]
		An argument using standard techniques (see, e.g., \cite[Lemma~4]{PlichkoStrictlySingular}) shows that for every $n$-dimensional subspace $V\subset c_0$, there exists an $x\in V$ with $\|x\|_{c_0}\le n^{-1/\max(p_0,q_0)}$. Combining these estimates yields the stated bound.
	\end{proof}
	
	\begin{cor}(\cite[Corollary 3.6]{ChianJanLidingI}). \label{Cor::SeqEmbed1}
		Let $s\ge t$, $0<p_0<p_1\le\infty$, and $0<q_0,q_1\le\infty$ such that either $s>t$ or $q_0<q_1$. Then the embedding
		\[
		\ell^{q_0}(2^{js}\cdot\ell^{p_0})_{j=0}^\infty\hookrightarrow\ell^{q_1}(2^{jt}\cdot\ell^{p_1})_{j=0}^\infty
		\]
		is finitely strictly singular (and in particular strictly singular).
	\end{cor}
	
	\begin{proof}
		When $s=t$, an isomorphism between $\ell^{q}(2^{js}\cdot\ell^{p})_{j=0}^\infty$ and $\ell^q(\ell^p)$ (given by multiplication by $2^{js}$) reduces the problem to the previous proposition.

    When $s>t$, we can decompose $\ell^{q_0}(2^{js}\cdot\ell^{p_0})_{j=0}^\infty\hookrightarrow\ell^{q_1}(2^{jt}\cdot\ell^{p_1})_{j=0}^\infty$ into
    \begin{equation*}
        \ell^{q_0}(2^{js}\cdot\ell^{p_0})_{j=0}^\infty \hookrightarrow 
        \ell^{\min(\frac{q_1}2,1)}(2^{jt}\cdot\ell^{p_0})_{j=0}^\infty\hookrightarrow\ell^{q_1}(2^{jt}\cdot\ell^{p_1})_{j=0}^\infty.
    \end{equation*}
    Here $\ell^{q_0}(2^{js}\cdot\ell^{p_0})_{j=0}^\infty\hookrightarrow\ell^{\min(q_1/2,1)}(2^{jt}\cdot\ell^{p_0})_{j=0}^\infty$ is bounded because by H\"older's inequality:
    \begin{align*}
        \|x\|_{\ell^{\min(q_1/2,1)}(2^{jt}\cdot\ell^{p_0})_{j=0}^\infty}^{\min(q_1/2,1)}=&\sum_{j=0}^\infty(2^{jt}\|x_j\|_{\ell^{p_0}})^{\min(q_1/2,1)} \\
        \le & \sup_{j\ge0}(2^{js}\|x_j\|_{\ell^{p_0}})^{\min(q_1/2,1)}\sum_{k=0}^\infty 2^{-k(s-t)\cdot \min(q_1/2,1)}
        \\\lesssim&_{s,t,q_1}\|x\|_{\ell^\infty(2^{js}\cdot\ell^{p_0})_{j=0}^\infty}^{\min(q_1/2,1)}\le \|x\|_{\ell^{q_0}(2^{js}\cdot\ell^{p_0})_{j=0}^\infty}^{\min(q_1/2,1)}.
    \end{align*}
    
    The result then follows immediately since $\ell^{\min(\frac{q_1}2,1)}(2^{jt}\cdot\ell^{p_0})_{j=0}^\infty\hookrightarrow\ell^{q_1}(2^{jt}\cdot\ell^{p_1})_{j=0}^\infty$ is already known to be finitely strictly singular.

	\end{proof}
	
	\begin{prop}(\cite[Proposition 3.8]{ChianJanLidingI}). \label{Prop::SeqEmbed2NotFSS}
		Let $0<p_1\le p_0\le\infty$ and $0<q_0<q_1\le\infty$, with $(p_0,p_1)\notin\{\infty\}\times(0,\infty)$. Then for every $A\ge0$, the embedding
		\begin{equation}\label{Eqn::SeqEmbed2NotFSS} 
			\ell^{q_0}\Big(2^{-j/{p_0}}\cdot\ell^{p_0}\{1,\dots,2^j\}\Big)_{j=A}^\infty\hookrightarrow\ell^{q_1}\Big(2^{-j/{p_1}}\cdot\ell^{p_1}\{1,\dots,2^j\}\Big)_{j=A}^\infty
		\end{equation}
		is not finitely strictly singular.
	\end{prop}
	
	\begin{proof}
		When $p_0=p_1$, we can take the $2^n$ dimensional subspace $V_n:=2^{-n/{p_0}}\cdot\ell^{p_0}\{1,\dots,2^n\}\hookrightarrow\ell^{q_0}(2^{-j/{p_0}}\cdot\ell^{p_0}\{1,\dots,2^j\})_{j=0}^\infty$, where $n\ge1$. The embedding \eqref{Eqn::SeqEmbed2NotFSS} restricted to $V_n$ is the identity map, hence the Bernstein numbers are all bounded from below by $1$, which does not tend to $0$.

    In the following we consider $0<p_0<p_1<\infty$. For each $n\ge1$ let $S_n:=(\{-1,1\}^n,\mu_n)$ be the probability space where $\mu_n\{a\}=2^{-n}$ for all $a\in\{-1,1\}^n$. For $1\le j\le n$, let $r_{n,j}:S_n\to\{-1,1\}$ be the standard projection on the $j$-th coordinate, i.e. $r_{n,j}(\epsilon_1,\dots,\epsilon_n)=\epsilon_j$. The standard Khintchine inequality yields $\|\sum_{j=1}^na_jr_{n,j}\|_{L^p(S_n,\mu_n)}\approx_p\|a\|_{\ell^2\{1,\dots,n\}}$ where the implied constant depends only on $p\in(0,\infty)$ but not on $n$. In particular, there is a $C_{p_0p_1}>0$ such that for every $n\ge1$ and $(a_j)_{j=1}^n\subset\R$,
    \begin{equation}\label{Eqn::KHIneqn}
        \Big\|\sum_{j=1}^na_jr_{n,j}\Big\|_{L^{p_1}(S_n,\mu_n)}\le \Big\|\sum_{j=1}^na_jr_{n,j}\Big\|_{L^{p_0}(S_n,\mu_n)}\le C_{p_0p_1}\Big\|\sum_{j=1}^na_jr_{n,j}\Big\|_{L^{p_1}(S_n,\mu_n)}.
    \end{equation}
    
    Since $(S_n,\mu_n)$ and $(\{1,\dots,2^n\},2^{-n}\cdot\#)$ are both probability spaces with uniform distributions on their sample sets, we see that they are isomorphic by taking any bijection $\Gamma_n :\{1,\dots,2^n\}\to S_n$ between sets. Thus we have elements $r_{n,j}\circ\Gamma_n\in\ell^{p_1}\{1,\dots,2^n\}$.
    
    Now take $X=\spa(r_{n,j}\circ\Gamma_n)_{j=1}^n$, we have $\dim X=n$. By \eqref{Eqn::KHIneqn} we get, for every $(a_j)_{j=1}^n\subset\R$,
    \begin{equation*}
        \Big\|\sum_{j=1}^na_j\cdot(r_{n,j}\circ\Gamma_n)\Big\|_{2^{-\frac n{p_0}}\cdot\ell^{p_0}}\le C_{p_0p_1}\Big\|\sum_{j=1}^na_j\cdot(r_{n,j}\circ\Gamma_n)\Big\|_{2^{-\frac n{p_1}}\cdot\ell^{p_1}}.
    \end{equation*}

    We conclude that $b_n\big(\ell^{q_0}(2^{-\frac j{p_0}}\cdot\ell^{p_0}\{1,\dots,2^j\})_{j=A}^\infty\hookrightarrow\ell^{q_1}(2^{-\frac j{p_1}}\cdot\ell^{p_1}\{1,\dots,2^j\})_{j=A}^\infty\big)\ge C_{p_0p_1}^{-1}$, which does not go to zero as $n\to\infty$.
	\end{proof}
	
	The above propositions lead to the following corollary.
	
	\begin{cor}(\cite[Corollary 3.9]{ChianJanLidingI}). \label{Cor::SeqEmbed2}
		Let $n,A\ge1$, $0<p_1\le p_0\le\infty$, and $0<q_0\le q_1\le\infty$. Consider the embedding
		\begin{equation}\label{Eqn::SeqEmbed2::Eqn}
			\ell^{q_0}\Big(2^{j(s-\frac{n}{p_0})}\cdot\ell^{p_0}\{1,\dots,2^{nj}\}\Big)_{j=A}^\infty\hookrightarrow\ell^{q_1}\Big(2^{j(s-\frac{n}{p_1})}\cdot\ell^{p_1}\{1,\dots,2^{nj}\}\Big)_{j=A}^\infty.
		\end{equation}
		\begin{enumerate}[(i)]
			\item \label{Item::SeqEmbed2::FSS} If $q_0<q_1$ and $p_1<p_0=\infty$, then the embedding is finitely strictly singular.
			\item \label{Item::SeqEmbed2::SS} If $q_0<q_1$ and either $p_1\le p_0<\infty$ or $p_0=p_1=\infty$, then the embedding is strictly singular but not finitely strictly singular.
			\item \label{Item::SeqEmbed2::NSS} If $q_0=q_1$, then the embedding is not strictly singular.
		\end{enumerate}
	\end{cor}
	
	\subsection{Embeddings between Besov Spaces}
	
	The results for sequence spaces, together with additional arguments, yield the following main theorems.
	
	\begin{thm}[{\cite[Theorem 1.1]{ChianJanLidingI}}] \label{Thm::ClassifyDom}
		Let $p_0,p_1,q_0,q_1\in(0,\infty]$ and $s_0,s_1\in\R$. Let $\Omega\subset\R^n$ be a bounded Lipschitz domain. Consider the embedding $\Bs_{p_0q_0}^{s_0}(\Omega)\hookrightarrow \Bs_{p_1q_1}^{s_1}(\Omega)$. Then:
		\begin{enumerate}[(i)]
			\item \label{Item::ClassifyDom::Not} There is no embedding (i.e. $\Bs_{p_0q_0}^{s_0}(\Omega)\not\subset \Bs_{p_1q_1}^{s_1}(\Omega)$) if and only if either
			\begin{itemize}
				\item $s_0-s_1 < \max\Big(0, \frac{n}{p_0}-\frac{n}{p_1}\Big)$, or
				\item $s_0-s_1 = \max\Big(0, \frac{n}{p_0}-\frac{n}{p_1}\Big)$ and $q_0 > q_1$.
			\end{itemize}
			\item \label{Item::ClassifyDom::Cpt} The embedding is compact if and only if 
			\[
			s_0-s_1 > \max\Big(0, \frac{n}{p_0}-\frac{n}{p_1}\Big).
			\]
			\item \label{Item::ClassifyDom::FSS} The embedding is non-compact but finitely strictly singular if and only if either
			\begin{itemize}
				\item $s_0-s_1 = \frac{n}{p_0}-\frac{n}{p_1} > 0$ and $q_0 < q_1$, or
				\item $s_0-s_1 = 0$, $p_1 < p_0 = \infty$, and $q_0 < q_1$.
			\end{itemize}
			\item \label{Item::ClassifyDom::SS} The embedding is strictly singular but not finitely strictly singular if and only if either
			\begin{itemize}
				\item $s_0-s_1 = 0$, $p_1 = p_0 = \infty$, and $q_0 < q_1$, or
				\item $s_0-s_1 = 0$, $p_1\le p_0<\infty$, and $q_0 < q_1$.
			\end{itemize}
			\item \label{Item::ClassifyDom::NSS} The embedding is not strictly singular if and only if
			\[
			s_0-s_1 = \max\Big(0, \frac{n}{p_0}-\frac{n}{p_1}\Big) \quad \text{and} \quad q_0=q_1.
			\]
		\end{enumerate}
	\end{thm}
	
	\begin{rmk*}
		By using a standard partition-of-unity argument, the results of Theorem~\ref{Thm::ClassifyDom} also hold when $\Omega$ is replaced by a compact smooth manifold.
	\end{rmk*}
	
	For embeddings between Besov spaces on $\R^n$, we have:
	
	\begin{thm}[{\cite[Theorem 1.2]{ChianJanLidingI}}]  \label{Thm::ClassifyRn}
		Let $p_0,p_1,q_0,q_1\in(0,\infty]$ and $s_0,s_1\in\R$. Consider the embedding $\Bs_{p_0q_0}^{s_0}(\R^n)\hookrightarrow \Bs_{p_1q_1}^{s_1}(\R^n)$. Then:
		\begin{enumerate}[(i)]
			\item \label{Item::ClassifyRn::Not} There is no embedding if and only if one of the following holds:
			\begin{itemize}
				\item $p_0 > p_1$,
				\item $s_0-s_1 < \frac{n}{p_0}-\frac{n}{p_1}$,
				\item $s_0-s_1 = \frac{n}{p_0}-\frac{n}{p_1} \ge 0$ and $q_0 > q_1$.
			\end{itemize}
			\item \label{Item::ClassifyRn::FSS} The embedding is non-compact but finitely strictly singular if and only if either
			\begin{itemize}
				\item $s_0-s_1 = \frac{n}{p_0}-\frac{n}{p_1} > 0$ and $q_0 < q_1$, or
				\item $s_0-s_1 > \frac{n}{p_0}-\frac{n}{p_1} > 0$.
			\end{itemize}
			\item \label{Item::ClassifyRn::NSS} The embedding is not strictly singular if and only if one of the following holds:
			\begin{itemize}
				\item $s_0-s_1 > 0$ and $p_0=p_1$,
				\item $s_0-s_1 = 0$, $p_0=p_1$, and $q_0\le q_1$, or
				\item $s_0-s_1 = \frac{n}{p_0}-\frac{n}{p_1} > 0$ and $q_0 = q_1$.
			\end{itemize}
		\end{enumerate}
	\end{thm}

\subsection{Sobolev Embeddings}

\begin{defn}[Non-compact Sobolev Embedding]
Let $m\in\Z_+$, $1\le p<\frac{d}{m}$, and set 
$
p^*=\frac{dp}{d-mp}.$
Let $\Omega\subseteq\R^n$ be an open subset (not necessarily bounded). Suppose that $0<q<r\le\infty$. In the case $p=1$, we additionally assume that $q=1$. Then the following non-compact Sobolev embedding holds:
\begin{equation}
E:W_0^m L^{p,q}(\Omega)\hookrightarrow L^{p^*,r}(\Omega).
\end{equation}
\end{defn}

From \cite{LMOP}, we know that $E$ is maximally non-compact. Moreover, from \cite{BouGro} and \cite{LangMihula} one obtains that:
\begin{itemize}
    \item $E$ is not strictly singular when $r=p$, and
    \item $E$ is finitely strictly singular when $r=p^*$.
\end{itemize}

A natural question that arises is what happens in the intermediate case when $p<r<p^*$. Unfortunately, the techniques used in \cite{BouGro} and \cite{LangMihula} do not apply in this range, so a new approach is required.

By using results on Besov embeddings, we obtain the following theorem, which answers the above question.

\begin{thm}[{\cite[Theorem 1]{ChianJanLidingII}}]\label{Thm::MainThm}
The non-compact Sobolev embedding
\[
E:W_0^m L^{p,q}(\Omega)\hookrightarrow L^{p^*,r}(\Omega)
\]
is finitely strictly singular when $p<r$. Moreover, the Bernstein numbers of this embedding satisfy the decay estimate: for every $\eps>0$, there exists $C_\eps>0$ such that
\begin{equation}\label{Eqn::MainBnEst}
b_n\Big(W_0^m L^{p,q}(\Omega)\hookrightarrow L^{p^*,r}(\Omega)\Big)
\le C_\eps\, n^{-\min\Big(\frac{1}{p+\eps},\frac{1}{q}\Big)
\Big(1-\max\Big(\frac{q}{r},\frac{p}{p^*}+\eps\Big)\Big)}.
\end{equation}
\end{thm}

When $p<q<r<p^*$, we obtain the sharp asymptotic estimate (see Proposition~\ref{Prop::SobSharp}):
\[
C^{-1} n^{-\Big(\frac{1}{q}-\frac{1}{r}\Big)}
\le b_n\Big(W_0^m L^{p,q}(\Omega)\hookrightarrow L^{p^*,r}(\Omega)\Big)
\le C\, n^{-\Big(\frac{1}{q}-\frac{1}{r}\Big)}.
\]

The next lemma is essential for the proof of the above theorem.

\begin{lem}[{\cite[Lemma 2]{ChianJanLidingII}}]\label{Lem::SobEmbed}
Let $m\in\Z_+$, $1\le p< \infty$, and $0<q\le\infty$. Then:
\begin{enumerate}[(i)]
    \item\label{Item::SobEmbed::1} When $1<p<\infty$, one has
    \[
    W^m L^{p,q}(\R^d)\hookrightarrow\Bs_{p^*,q}^s(\R^d)
    \]
    for any $s<m$ and for every $p^*\in(p,\infty]$ satisfying 
    \[
    \frac{1}{p^*}=\frac{1}{p}-\frac{m-s}{d}.
    \]
    \item\label{Item::SobEmbed::2} When $p=1$, 
    \[
    W^{m,1}(\R^d) \hookrightarrow \Bs_{p^*,1}^s(\R^d)
    \]
    for $m-d\le s<m$ and $\frac{1}{p^*}=1-\frac{m-s}{d}$.
    \item\label{Item::SobEmbed::3} When $1<p<\infty$, 
    \[
    \Bs_{\tilde p,q}^s(\R^d)\hookrightarrow L^{p,q}(\R^d)
    \]
    for any $s>0$ and with 
    \[
    \frac{1}{\tilde p}=\frac{1}{p}+\frac{s}{d}.
    \]
\end{enumerate}
\end{lem}

\begin{proof}
Parts (\ref{Item::SobEmbed::1}) and (\ref{Item::SobEmbed::3}) can be obtained by real interpolation using classical Sobolev embeddings. In particular, for $0<p,p_0,p_1 <\infty$, $s_0,s_1\in\R$, $0<q,q_0,q_1\le\infty$, and $0<\theta<1$ with $p_0\neq p_1$, if
\[
\frac{1-\theta}{p_0}+\frac{\theta}{p_1}=\frac{1}{p_\theta} \quad\text{and}\quad s_\theta=(1-\theta)s_0+\theta s_1,
\]
then
\begin{equation}
\label{Eqn::SobEmbed::RealInterpo}
(L^{p_0}(\R^d),L^{p_1}(\R^d))_{\theta,q}=L^{p_\theta,q}(\R^d)
\quad \text{and} \quad
(\Bs_{p,q_0}^{s_0}(\R^d),\Bs_{p,q_1}^{s_1}(\R^d))_{\theta,q}=\Bs_{p,q}^{s_\theta}(\R^d).
\end{equation}
See \cite[Theorem~4.3]{HolmstedtInterpolation} and \cite[Theorem~4.25]{SawanoBook}, respectively.

On the other hand, for $r>0$, we have the embeddings
\begin{gather}
\label{Eqn::SobEmbed::Lp1}
L^{\frac{dp^*}{d+rp^*}}(\R^d)\hookrightarrow\Bs_{p^*,p^*}^{-r}(\R^d),\qquad \frac{d}{d-r}<p^*\le\infty,
\\
\label{Eqn::SobEmbed::Lp2}
\Bs_{\tilde p,\tilde p}^{r}(\R^d)\hookrightarrow L^{\frac{d\tilde p}{d-r\tilde p}}(\R^d),\qquad \frac{d}{d+r}<\tilde p<\frac{d}{r}.
\end{gather}
See, e.g., \cite[Remark~2.7.1/3]{TriebelTheoryOfFunctionSpacesI} (note that in that reference, $L^p(\R^d)$ is identified with $\Fs_{p2}^0(\R^d)$ by \cite[Theorem~2.5.6]{TriebelTheoryOfFunctionSpacesI}).

Thus, by applying \eqref{Eqn::SobEmbed::RealInterpo} and \eqref{Eqn::SobEmbed::Lp1} with appropriate parameters (choosing $r_0<r_\theta=m-s<r_1$), we obtain part (\ref{Item::SobEmbed::1}).

Similarly, applying \eqref{Eqn::SobEmbed::RealInterpo} and \eqref{Eqn::SobEmbed::Lp2} with $r_0<r_\theta=s<r_1$, we obtain part (\ref{Item::SobEmbed::3}).

When $p=1$, by \cite[(1.2)]{SpectorL1} we have that 
\[
(-\Delta)^{-\frac{r}{2}}W^{1,1} (\R^d)\to W^1L^{\frac{d}{d-r},1}(\R^d)
\]
for $0<r<d$ (this refines the result in \cite{SchikorraSpectorVanSchaftingenL1}, which shows that $(-\Delta)^{-\frac{r}{2}}W^{1,1}\to W^{1,\frac{d}{d-r}}$). Using the H\"ormander-Mikhlin multiplier theorem (since $1<\frac{d}{d-r}<\infty$), one obtains that
\[
(I-\Delta)^{-\frac{r}{2}}(-\Delta)^{\frac{r}{2}}:L^{\frac{d}{d-r},1}\to L^{\frac{d}{d-r},1},
\]
i.e., $(I-\Delta)^{-\frac{r}{2}}:W^{1,1} (\R^d)\to W^1L^{\frac{d}{d-r},1}(\R^d)$. Then, applying part (\ref{Item::SobEmbed::1}), we deduce that 
\[
W^{1,1}(\R^d)\hookrightarrow\Bs_{\frac{d}{d-2r},1}^{\,1-r}(\R^d)
\]
for $0<r\le\frac{d}{2}$. Taking $r=\frac{m-s}{2}$ and using the relevant norm equivalences yields part (\ref{Item::SobEmbed::2}).
\end{proof}

\begin{proof}[Proof of Theorem~\ref{Thm::MainThm}]
Note that $W^m_0L^{p,q}(\Omega)$ is a closed subspace of $W^mL^{p,q}(\R^d)$ (with the same Sobolev norm) for any open set $\Omega\subseteq\R^d$. Therefore,
\[
b_n\Big(W^m_0L^{p,q}(\Omega)\hookrightarrow L^{p^*,r}(\Omega)\Big)
\le b_n\Big(W^m_0L^{p,q}(\R^d)\hookrightarrow L^{p^*,r}(\R^d)\Big),
\]
so it suffices to consider the case $\Omega=\R^d$.

Let $\delta>0$ be a small number. By Lemma~\ref{Lem::SobEmbed}, there exists a factorization
\[
W^mL^{p,q}(\R^d)\hookrightarrow\Bs_{p+\delta,q}^{\,m-\frac{d\delta}{p(p+\delta)}}(\R^d)
\overset{i}{\hookrightarrow}\Bs_{p^*-\delta,r}^{\,\frac{d\delta}{p^*(p^*-\delta)}}(\R^d)
\hookrightarrow L^{p^*,r}(\R^d),
\]
where $\overset{i}{\hookrightarrow}$ is an embedding between Besov spaces (see \cite{ChianJanLidingI}).
Using the wavelet decomposition (see Proposition \ref{Prop::Wavelet} and \cite[Proposition~2.13]{ChianJanLidingI}), there exists an isomorphism
\[
\Lambda:\Bs_{p^*-\delta,r}^{\,\frac{d\delta}{p^*(p^*-\delta)}}(\R^d)\xrightarrow{\simeq}\ell^r\Big(\ell^{p^*-\delta}\Big)
\]
such that the restriction
\[
\Lambda:\Bs_{p+\delta,q}^{\,m-\frac{d\delta}{p(p+\delta)}}(\R^d)\xrightarrow{\simeq}\ell^q\Big(\ell^{p+\delta}\Big)
\]
is also an isomorphism. By composing these isomorphisms, one deduces that

\begin{align*}
    b_n\Big(W^mL^{p,q}(\R^d)\hookrightarrow L^{p^*,r}(\R^d)\Big)
& \lesssim_\delta b_n\Big(\Bs_{p+\delta,q}^{\,m-\frac{d\delta}{p(p+\delta)}}(\R^d)
\hookrightarrow\Bs_{p^*-\delta,r}^{\,\frac{d\delta}{p^*(p^*-\delta)}}(\R^d)\Big) \\
& \lesssim_\delta b_n\Big(\ell^q\big(\ell^{p+\delta}\big)
\hookrightarrow\ell^r\big(\ell^{p^*-\delta}\big)\Big).
\end{align*}

Conversely, by Proposition \ref{Prop::SeqEmbed1} (see also  \cite[Proposition~3.5]{ChianJanLidingI}), we have
\[
b_n\Big(\ell^q\big(\ell^{p+\delta}\big)
\hookrightarrow\ell^r\big(\ell^{p^*-\delta}\big)\Big)
\le n^{-\min\Big(\frac{1}{p+\delta},\frac{1}{q}\Big)
\Big(1-\max\Big(\frac{q}{r},\frac{p+\delta}{p^*-\delta}\Big)\Big)}.
\]
Now, for $\eps>0$ given in the statement, choose $0<\delta<\eps$ so that
\[
\frac{p+\delta}{p^*-\delta}<\frac{p}{p^*}+\eps.
\]
This yields the desired estimate \eqref{Eqn::MainBnEst} and completes the proof.
\end{proof}

Note that when $p<q<r<p^*$, for $\delta>0$ sufficiently small, we have 
\[
\min\Big(\frac{1}{p+\delta},\frac{1}{q}\Big)\Big(1-\max\Big(\frac{q}{r},\frac{p+\delta}{p^*-\delta}\Big)\Big)
=\frac{1}{q}-\frac{1}{r}.
\]
This exponent is sharp, as shown by the following proposition.

\begin{prop}\label{Prop::SobSharp}
Let $m\ge1$, $1\le p<\frac{d}{m}$, and $0<q\le r\le\infty$ (with $q\le1$ if $p=1$). Define $p^*=\frac{dp}{d-mp}$. Then there exists a constant $c>0$ such that
\[
b_n\Big(W_0^mL^{p,q}(\Omega)\hookrightarrow L^{p^*,r}(\Omega)\Big)
\ge c\, n^{\frac{1}{r}-\frac{1}{q}},\qquad n\ge1.
\]
Moreover, when $q=r$, the embedding
\[
W_0^mL^{p,q}(\Omega)\hookrightarrow L^{p^*,r}(\Omega)
\]
is not strictly singular.
\end{prop}

\begin{proof}
See Proposition 5 in \cite{ChianJanLidingII} for the detailed proof.
\end{proof}

\section{Fourier Transformation}

Results in this section are based on \cite{EdmundsGurkaLang}.

For a function \( f \in L^1(\R^d) \) we define its \emph{Fourier transform} by
\begin{equation}\label{FourTrans}
	\FT(f)(x)
	=
	\widehat{f}(x)
	=
	\int_{\R^d} f(t) \, e^{-i\langle x,t\rangle} \, \dd t,
\end{equation}
where for \( x=(x_1,\dots,x_d) \) and \( t=(t_1,\dots,t_d) \) we have 
\[
\langle x,t\rangle = \sum_{k=1}^{d} x_k t_k.
\]

We recall some basic properties of the Fourier transform:
\begin{enumerate}[{\rm(i)}]
	\item \textbf{Linearity:} \quad 
	\(\FT(\alpha f+\beta g)=\alpha\, \FT(f)+\beta\, \FT(g)\) (where \(\alpha\) and \(\beta\) are constants);
	\item \textbf{Scaling:} \quad 
	\(\FT\big(f(c\,t)\big)(x)= c^{-d}\,\widehat{f}\big(x/c\big)\) (with \( c>0 \) constant).
\end{enumerate}

It is immediate that
\[
|\FT(f)(x)| \le \int_{\R^d} |f(t)| \, \dd t,
\]
and by Parseval's formula, we have
\[
\int_{\R^d} |\FT(f)(x)|^2 \, \dd x = (2\pi)^d \int_{\R^d} |f(x)|^2 \, \dd x.
\]
Combining this with the Riesz–Thorin interpolation theorem, one obtains
\[
\FT: L^p(\R^d) \to L^{p'}(\R^d), \quad \text{for } 1 \le p \le 2,
\]
(see, e.g., \cite[Section 1.2]{MR482275} for more details).

The following lemma will be needed later.

\begin{lem}[{\cite[Lemma 2.3]{EdmundsGurkaLang}}]\label{DilatL}
Let \( F: \R^d \to \R \) be a measurable function. For \( c>0 \) denote
\[
F_c(x)=F(c\,x), \quad x\in\R^d.
\]
Then, for all \( s>0 \),
\[
(F_c)^*(s)= F^*(c^d\,s).
\]
\end{lem}

\begin{proof}
Assume that \( c,\, t>0 \). Then
\[
(F_c)_{*}(t)
=
\mu\{x\in\R^d : |F_c(x)|>t\}
=
\mu\{x\in\R^d : |F(c x)|>t\}.
\]
Performing the change of variables \( y = c x \), we obtain
\[
\mu\{x\in\R^d : |F(c x)|>t\}
=
\int_{\{y\in\R^d:|F(y)|>t\}} c^{-d} \, \dd y
= c^{-d} F_{*}(t).
\]
Thus,
\[
(F_c)^*(s)
=
\inf\{t>0 : (F_c)_{*}(t) \le s\}
=
\inf\{t>0 : F_{*}(t) \le c^d s\}
=
F^*(c^d s).
\]
\end{proof}

We first recall some known results for the convenience of the reader.

\begin{thm}[{\cite[Th. 3.1]{EdmundsGurkaLang}}]\label{BddFT}
The Fourier transform \( \FT \) is a bounded linear operator from \( L^p(\R^d) \) into \( L^{p',p}(\R^d) \), i.e.,
\begin{equation}\label{spojFTdoLor}
   \FT: L^p(\R^d) \to L^{p',p}(\R^d),
\end{equation}
provided that \( p \in (1,2) \).
\end{thm}

\begin{proof}
We prove the assertion by interpolation. It is clear that the Fourier transform satisfies
\[
\FT: L^1(\R^d) \to L^\infty(\R^d) \quad \text{and} \quad \FT: L^2(\R^d) \to L^2(\R^d).
\]
Applying the Marcinkiewicz interpolation theorem (see, e.g., \cite[Theorem 4.13]{BS}), we obtain
\[
\FT: L^{p,r}(\R^d) \to L^{p',r}(\R^d)
\]
for every \( r \in [1,\infty] \). In particular, setting \( r = p \) and noting that \( L^p(\R^d) = L^{p,p}(\R^d) \), the result follows.
\end{proof}

\begin{rem}
If one applies the Riesz–Thorin interpolation theorem, one obtains the boundedness
\[
\FT: L^p(\R^d) \to L^{p'}(\R^d)
\]
for \( 1 \le p \le 2 \) (see, e.g., \cite[Theorem 1.2.1]{MR482275}). For \( p \in (1,2) \), this result is weaker than \eqref{spojFTdoLor} because \( L^{p',p}(\R^d) \hookrightarrow L^{p'}(\R^d) \).
\end{rem}

It follows from \cite[Theorem 5.1]{LefevrePiazzaFSSApplication} that the Fourier transform
\[
\FT: L^p(\R^d) \to L^{p'}(\R^d)
\]
is finitely strictly singular if and only if \( p \in (1,2) \).

In the following, which is the main result in  \cite{EdmundsGurkaLang} and where is possible to find more detailed proof, we prove that the Fourier transform from \( L^p \) into \( L^{p',p} \) is not strictly singular, which suggest that \( L^{p',p} \) is the optimal target space.

\begin{thm}[{\cite[Theorem 3.3]{EdmundsGurkaLang}}]\label{MainThhm}
Let \( p \in (1,2) \). Then the Fourier transform
\[
\FT: L^p(\R^d) \to L^{p',p}(\R^d)
\]
is not strictly singular.
\end{thm}

This result contrasts with the fact that the mapping
\[
\FT: L^p(\R^d) \rightarrow L^{p'}(\R^d)
\]
is finitely strictly singular when \( 1 < p < 2 \) (see \cite{LefevrePiazzaFSSApplication}).

\subsection{Proof of Theorem \ref{MainThhm}}

For simplicity, we provide a detailed proof in the one-dimensional case (\( d=1 \)); the proof in higher dimensions is analogous.

We begin by introducing suitable \emph{test functions}. For a fixed \( a>0 \) define
\begin{equation}\label{Funct_fa}
  f_a(t)=\chi_{[-1/a,1/a]}(t),\quad t\in\R.
\end{equation}
It is easy to verify that
\[
\FT f_a(x)=\frac{2\sin(x/a)}{x},\quad x\in\R.
\]

\begin{lem}\label{FToffa}
Let \( a>0 \) and \( p>1 \), and let \( f_a \) be the function defined in \eqref{Funct_fa}. Define
\[
c_p=2^{-1/p}\|\FT f_1\|_{p',p}.
\]
Then
\begin{equation}\label{NormsOFfa}
  \|f_a\|_p=(2/a)^{1/p},\quad
  \|\FT f_a\|_{p',p}=c_p\,(2/a)^{1/p}.
\end{equation}
\end{lem}

\begin{proof}
The first equality is immediate. For the second, observe that
\[
f_a(t)=f_1(at),\quad t\in\R.
\]
Using linearity and the scaling property of the Fourier transform together with Lemma~\ref{DilatL}, we obtain
\[
(\FT f_a)^*(\tau)
=
a^{-1}\big(\FT(f_1(x/a))\big)^*(\tau)
=
a^{-1}(\FT f_1)^*(\tau/a),\quad \tau>0.
\]
Making the change of variables \( y=a\tau \) leads to
  \begin{align*}
\|\FT f_a\|_{p',p}^p &
=
\int_{0}^{\infty}\Big(\tau^{1/p'}\, a^{-1}\, (\FT f_1)^*(\tau/a)\Big)^p \frac{\dd \tau}{\tau}
\\
& =
a^{-1}\int_{0}^{\infty}\Big(y^{1/p'}\, (\FT f_1)^*(y)\Big)^p\frac{\dd y}{y}
=
a^{-1}\|\FT f_1\|_{p',p}^p.
  \end{align*}
Thus, \(\|\FT f_a\|_{p',p} = c_p\,(2/a)^{1/p}\).
\end{proof}

\paragraph{\bf Basic Test Function.}\quad
Define
\begin{equation}\label{BasicTF}
  g_a(t)=(2/a)^{-1/p}f_a(t),\quad t\in\R, a>0.
\end{equation}
Then, by \eqref{NormsOFfa},
\begin{equation}\label{NormsOFga}
  \|g_a\|_p=1,\quad
  \|\FT g_a\|_{p',p}=c_p.
\end{equation}

Next, for a given \(\gamma>0\), by the absolute continuity of the norms \(\|\cdot\|_p\) and \(\|\cdot\|_{p',p}\), one can choose numbers \(\eta\in(0,1)\) and \(\nu^L, \nu^R\) with \(0<\nu^L<\nu^R<\infty\) such that
\begin{equation}\label{EstwithEP}
  \|g_1\,\chi_{[-1,-\eta)\cup(\eta,1]}\|_p\ge1-\gamma,
\end{equation}
and
\begin{equation}\label{EstwithEP02}
  \|(\FT g_1)\,\chi_{[-\nu^R,-\nu^L)\cup(\nu^L,\nu^R]}\|_{p',p}\ge c_p\,(1-\gamma/2).
\end{equation}
Using the definition of the Lorentz norm, we have
\begin{align*}
    \|(\FT g_1)\,\chi_{[-\nu^R,-\nu^L)\cup(\nu^L,\nu^R]}\|_{p',p}
& =  \\
 \Biggl(\int_{0}^{2(\nu^R-\nu^L)}\Bigl(y^{1/p'} & \big((\FT g_1\,\chi_{[-\nu^R,-\nu^L)\cup(\nu^L,\nu^R]})^*(y)\bigr)\Bigr)^p\frac{\dd y}{y}\Biggr)^{1/p}.
\end{align*}

By the absolute continuity of the integral and \eqref{EstwithEP02}, there exist numbers \(\delta^L\) and \(\delta^R\) with \(0<\delta^L<\delta^R<2(\nu^R-\nu^L)\) such that
\begin{equation}\label{EstwithEP03}
  \Biggl(\int_{\delta^L}^{\delta^R}\Bigl(y^{1/p'}\big((\FT g_1\,\chi_{[-\nu^R,-\nu^L)\cup(\nu^L,\nu^R]})^*(y)\bigr)^p\frac{\dd y}{y}\Biggr)^{1/p}
  \ge c_p\,(1-\gamma).
\end{equation}
By the same scaling argument as in the proof of Lemma~\ref{FToffa}, for any \( a>0 \) we have
\begin{align}\label{ScaledTestF01}
 & \|g_a\,\chi_{[-1/a,-\eta/a)\cup(\eta/a,1/a]}\|_p\ge1-\gamma, \\
 \label{ScaledTestF02}
 &  \Biggl(\int_{\delta^L a}^{\delta^R a}\Bigl(y^{1/p'}\big((\FT g_a\,
 \chi_{[-\nu^R a,-\nu^L a)\cup(\nu^L a,\nu^R a]})^*(y)\bigr)^p\frac{\dd y}{y}\Biggr)^{1/p}
 \ge c_p\,(1-\gamma).
\end{align}

\paragraph{\bf Construction of an Infinite-Dimensional Subspace.}\quad
Let \(\varepsilon>0\) be fixed. For \(\gamma=\varepsilon/2\) choose numbers \(\eta\), \(\nu^L\), and \(\nu^R\) such that \eqref{EstwithEP} and \eqref{EstwithEP03} hold. Then \eqref{ScaledTestF01} and \eqref{ScaledTestF02} hold with \(\gamma=\varepsilon/2\), and we denote these parameters by \(\eta_1\), \(\nu^L_1\), \(\nu^R_1\), \(\delta^L_1\), and \(\delta^R_1\) with \(a_1=1\). 

In the next step, with \(\gamma=\varepsilon/2^2\) we choose numbers \(\eta_2 \in (0,\eta_1)\), \(\nu^L_2 \in (0,\nu^L_1)\), and \(\nu^R_2 \in (\nu^R_1,\infty)\) so that the corresponding inequalities hold with parameters \(\eta_2\), \(\nu^L_2\), \(\nu^R_2\) and with new scaling factor \(a_2\) satisfying
\[
a_{2}>a_1\max\{1/\eta_1,\,\nu^R_1/\nu^L_2,\,\delta^R_1/\delta^L_2\}.
\]
This choice guarantees that
\[
\big([-1/a_1,-\eta_1/a_1)\cup(\eta_1/a_1,1/a_1)\big)
\cap
\big([-1/a_2,-\eta_2/a_2)\cup(\eta_2/a_2,1/a_2)\big)
=\emptyset,
\]
and similarly,
\[
\big([-\nu^R_1 a_1,-\nu^L_1 a_1)\cup(\nu^L_1 a_1,\nu^R_1 a_1]\big)
\cap
\big([-\nu^R_2 a_2,-\nu^L_2 a_2)\cup(\nu^L_2 a_2,\nu^R_2 a_2]\big)
=\emptyset,
\]
as well as
\[
(a_1\delta^L_1, a_1\delta^R_1) \cap (a_2\delta^L_2, a_2\delta^R_2) = \emptyset.
\]
Proceeding inductively, we obtain sequences \(\{\eta_j\}_{j=1}^\infty\), \(\{\nu^L_j\}_{j=1}^\infty\), \(\{\nu^R_j\}_{j=1}^\infty\), \(\{\delta^L_j\}_{j=1}^\infty\), \(\{\delta^R_j\}_{j=1}^\infty\), and \(\{a_j\}_{j=1}^\infty\) and define the functions
\begin{equation}\label{SystOFtestFu}
  \varphi_j = g_{a_j}, \quad j\in\Ne.
\end{equation}
These functions satisfy
\begin{align}\label{FinalTestF01}
  \|\varphi_j\,\chi_{G_j}\|_p &\ge 1-\varepsilon 2^{-j}, \\
 \label{FinalTestF02}
  \Biggl(\int_{I_j}\Bigl(y^{1/p'}\big((\FT \varphi_j\,\chi_{\widehat{G}_j})^*(y)\bigr)^p\frac{\dd y}{y}\Biggr)^{1/p} &\ge c_p(1-\varepsilon 2^{-j}),
\end{align}
where the sets
\[
G_j = \Bigl[-\frac{1}{a_j},-\frac{\eta_j}{a_j}\Bigr) \cup \Bigl(\frac{\eta_j}{a_j},\frac{1}{a_j}\Bigr], \quad
\widehat{G}_j = \Bigl[-\nu^R_j a_j, -\nu^L_j a_j\Bigr) \cup \Bigl(\nu^L_j a_j, \nu^R_j a_j\Bigr],
\]
and
\[
I_j = (a_j\delta^L_j, a_j\delta^R_j),
\]
satisfy
\begin{equation}\label{SystDisj}
  G_j \cap G_k = \emptyset,\quad \widehat{G}_j \cap \widehat{G}_k = \emptyset,\quad
  I_j \cap I_k = \emptyset,
\end{equation}
for all \( j \neq k \).

\begin{proof}[Proof of Theorem~\ref{MainThhm}]
\textbf{Case \( d=1 \).}\quad
Following the definition~ of strict singular operators, we show that there exists an infinite-dimensional subspace \( X \subset L^p(\R) \) and a constant \( b>0 \) such that for any \( f\in X \) with \( \|f\|_p=1 \), we have
\[
\|\FT f\|_{p',p} \ge b.
\]

Set $\varepsilon>0$ and consider the corresponding infinite-dimensional closed subspace 
\[
X = \operatorname{span}\big\{ \varphi_j : j\in\Ne \big\} \subset L^p(\R),
\]
where the sequence \( \{\varphi_j\} \) is defined in \eqref{SystOFtestFu}. Any function \( f \in X \) can be written as
\begin{equation}\label{Funct_f}
  f = \sum_{j=1}^{\infty} \alpha_j \varphi_j, \quad \alpha_j \in \R.
\end{equation}

We now estimate \( \|f\|_p \) from above and \( \|\FT f\|_{p',p} \) from below.

\textbf{Upper estimate for \( \|f\|_p \):}  
Denote 
\[
E_j = \Bigl[-\frac{\eta_j}{a_j},\frac{\eta_j}{a_j}\Bigr] = \Bigl[-\frac{1}{a_j},\frac{1}{a_j}\Bigr] \setminus G_j, \quad j\in\Ne.
\]
Then clearly,
\[
\|\varphi_j \chi_{E_j}\|_p \le \varepsilon\, 2^{-j}.
\]
By the Minkowski inequality,
\[
\|f\|_p
\le \Bigl\|\sum_{j=1}^{\infty} \alpha_j \varphi_j \chi_{G_j}\Bigr\|_p
+ \Bigl\|\sum_{j=1}^{\infty} \alpha_j \varphi_j \chi_{E_j}\Bigr\|_p.
\]
Since the supports of \( \varphi_j \chi_{G_j} \) are disjoint (by \eqref{SystDisj}), we have
\[
\Bigl\|\sum_{j=1}^{\infty} \alpha_j \varphi_j \chi_{G_j}\Bigr\|_p^p
=
\sum_{j=1}^{\infty} |\alpha_j|^p \|\varphi_j \chi_{G_j}\|_p^p
\le \sum_{j=1}^{\infty} |\alpha_j|^p,
\]
and, using \( |\alpha_j| \le A \) (where \( A = \bigl(\sum_{j=1}^{\infty} |\alpha_j|^p\bigr)^{1/p} \)),
\[
\Bigl\|\sum_{j=1}^{\infty} \alpha_j \varphi_j \chi_{E_j}\Bigr\|_p
\le A \sum_{j=1}^{\infty} \|\varphi_j \chi_{E_j}\|_p
\le A \sum_{j=1}^{\infty} \frac{\varepsilon}{2^{j}} = A\varepsilon.
\]
Hence,
\begin{equation}\label{UpperEst}
\|f\|_p \le A(1+\varepsilon).
\end{equation}

\textbf{Lower estimate for \( \|\FT f\|_{p',p} \):}  
We have
\[
\|\FT f\|_{p',p} = \Bigl\|\sum_{j=1}^{\infty} \alpha_j \FT \varphi_j \Bigr\|_{p',p}
\ge \Bigl\|\sum_{j=1}^{\infty} \alpha_j \chi_{\widehat{G}_j} \FT \varphi_j \Bigr\|_{p',p}.
\]
Since the functions \( \chi_{\widehat{G}_j} \FT \varphi_j \) have disjoint supports (by \eqref{SystDisj}), Lemma~\ref{SumNonIncrRear} implies that for all \( t>0 \),
\[
\Bigl(\sum_{j=1}^{\infty} \alpha_j \chi_{\widehat{G}_j} \FT \varphi_j \Bigr)^*(t)
\ge \sum_{j=1}^{\infty} |\alpha_j|\, \chi_{I_j}(t) \Bigl(\chi_{\widehat{G}_j} \FT \varphi_j\Bigr)^*(t).
\]
Thus,  using above estimates and 
\eqref{SystDisj} and \eqref{FinalTestF02} we have we have
\begin{multline}\label{est01}
\|\FT f\|_{p',p}^p
=
\int_{0}^\infty \Bigl(t^{1/p'} \Bigl(\sum_{j=1}^{\infty} \alpha_j \chi_{\widehat{G}_j} \FT \varphi_j \Bigr)^*(t)\Bigr)^p \frac{\dd t}{t} \\
\ge \sum_{j=1}^{\infty} |\alpha_j|^p
\int_{I_j} \Bigl(t^{1/p'} \bigl((\FT \varphi_j \, \chi_{\widehat{G}_j})^*(t)\bigr)\Bigr)^p \frac{\dd t}{t}
\ge \sum_{j=1}^{\infty} |\alpha_j|^p \bigl(c_p(1-\varepsilon2^{-j})\bigr)^p.
\end{multline}
That is,
\begin{equation}\label{LowerEst}
\|\FT f\|_{p',p} \ge \Bigl(\sum_{j=1}^{\infty} |\alpha_j|^p\Bigr)^{1/p} c_p(1-\varepsilon)
= A\, c_p (1-\varepsilon).
\end{equation}

Combining the upper estimate \eqref{UpperEst} with the lower estimate \eqref{LowerEst} shows that the Fourier transform, restricted to the subspace \( X \), is bounded below by a constant independent of the dimension of \( X \). This completes the proof for \( d=1 \).

\textbf{Case \( d>1 \):}  
In this case, consider the cube
\[
Q_a = \Bigl[-\frac{1}{a},\frac{1}{a}\Bigr]^d,
\]
and define the test function
\[
g_a(t) = (2/a)^{-d/p}\chi_{Q_a}(t), \quad t\in\R^d.
\]
Then one easily verifies that
\[
\|g_a\|_p=1 \quad \text{and} \quad \|\FT g_a\|_{p',p}=c_{d,p},
\]
where \( c_{d,p} = \|\FT g_1\|_{p',p} \). The proof in the \( d>1 \) case follows analogously, and we omit the details.
\end{proof}

The above result can be modified to prove the following theorems, which illustrate the optimal source and target spaces for the Fourier transform. For more details see \cite{EdmundsGurkaLang}.

\begin{thm}[{\cite[Theorem 5.6]{EdmundsGurkaLang}}]\label{MainDiscreteThm}
Let \( p\in(1,2) \). The discrete Fourier transform is a bounded linear operator from \( L^p(\T^d) \) into the sequence space \( \ell^{p',p}(\Z^d) \); however, it is not strictly singular.
\end{thm}

\section{Laplace Transformation}

Results in this sections are taken from \cite{EdmundsLangLaplace}.
The Laplace transform is a classical integral operator \( L \) defined for a suitable function \( f \) on \( (0, \infty) \) by  
\[
Lf(t) = \int_{0}^{\infty} f(s) e^{-ts} \, ds, \quad t \in (0, \infty).
\]  
The boundedness of \( L \) when acting between various function spaces has been widely investigated. It is well known that \( L \) is bounded from the Lebesgue space \( L^p(0, \infty) \) into \( L^{p'}(0, \infty) \), where 
$
p' = \frac{p}{p-1},
$
provided \( 1 \leq p < 2 \). However, when \( p > 2 \), there does not exist a Lebesgue space \( L^q \) that can serve as the target space for the Laplace transform.

To address this, the framework of Lorentz spaces \( L^{p,q} \) becomes essential. Lorentz spaces extend the classical Lebesgue spaces and allow for a more refined analysis. It has been shown that if \( p \in (1, \infty) \) and \( q \in [1, \infty] \), then  
\[
L : L^{p,q}(0, \infty) \rightarrow L^{p',q}(0, \infty)
\]  
is a bounded operator. Furthermore, within the category of rearrangement-invariant spaces, the domain \( L^{p,q}(0, \infty) \) and the target space \( L^{p',q}(0, \infty) \) form an optimal pair (see Buriánková, Edmunds, and Pick \cite{BEP2017} for details).

While boundedness is well established, it is natural to investigate whether
\[
L : L^{p,q}(0, \infty) \to L^{p',q}(0, \infty)
\]
possesses any properties stronger than boundedness, such as compactness. In this paper, we demonstrate that, in this setting, the Laplace transform fails to be compact. More significantly, we show that it is maximally non-compact and not strictly singular. These findings provide a deeper understanding of the functional-analytic properties of the Laplace transform in Lorentz spaces. We refer to \cite{EdmundsLangLaplace} for more details from which the following results are taken.

\subsection{Laplace Transformation --- Main Results}

Our objective is to prove the following statements:

\begin{theorem}[{\cite[Theorem 1]{EdmundsLangLaplace}}] \label{Thereom 1}
	Let \( p, q, r \in (1,\infty) \), with \( p' = \frac{p}{p-1} \) and \( 1 < q \leq r < \infty \). Then
	\[
	L : L^{p,q}(0,\infty) \to L^{p',r}(0,\infty)
	\]
	is maximally non-compact.
\end{theorem}

\begin{theorem}[{\cite[Theorem 2]{EdmundsLangLaplace}}] \label{Thereom 2}
	Let \( p, q \in (1,\infty) \) and let \( p' = \frac{p}{p-1} \). Then
	\[
	L : L^{p,q}(0,\infty) \to L^{p',q}(0,\infty)
	\]
	is not strictly singular.
\end{theorem}

We establish the above results using a series of auxiliary lemmas and observations. For clarity, we now present some foundational lemmas, including proofs where necessary.

\begin{lemma} \label{Lemma 3.4}
	Let \( 1 \leq q < \infty \), \( \varepsilon > 0 \), and let \( \{f_i\}_{i=1}^\infty \) and \( \{g_i\}_{i=1}^\infty \) be sequences in a Banach space \( X \) satisfying 
	\[
	\|f_i - g_i\| \leq \varepsilon\, 2^{-i},
	\]
	and let \( \{\alpha_i\}_{i=1}^\infty \in \ell_{\infty} \). Then
	\[
	\left\| \sum_{i=1}^\infty \alpha_i g_i \right\|_X - 2\varepsilon\, \max_i |\alpha_i| \leq \left\| \sum_{i=1}^\infty \alpha_i f_i \right\|_X \leq \left\| \sum_{i=1}^\infty \alpha_i g_i \right\|_X + 2\varepsilon\, \max_i |\alpha_i|.
	\]
\end{lemma}

\begin{proof}
	This follows directly from the triangle inequality. (The details are left as an exercise for the reader.)
\end{proof}

The following observations provide the groundwork for the proofs of the main theorems.

Assume that \( 1 < p < \infty \) and \( 1 \leq q \leq r < \infty \). Define
\[
c_1 := \|L : L^{p,q}(0,\infty) \to L^{p',r}(0,\infty)\|.
\]
For each \( \varepsilon > 0 \), there exists a function \( f_{\varepsilon} \) with \( \|f_{\varepsilon}\|_{p,q} = 1 \) such that
\[
\|L f_{\varepsilon}\|_{p',r} \geq c_1 - \varepsilon.
\]

Using the fact that \( \|f\|_{L^{p,q}(0,\infty)} = \|f^*\|_{L^{p,q}(0,\infty)} \) and the definition of \( L \), we have
\begin{eqnarray}
	&& L(f)(t) \leq L(|f|)(t) \leq L(f^*)(t), \nonumber \\
	&& \text{and} \quad (L(|f|))^*(t) = L(|f|)(t) \quad \text{for } t \in (0,\infty).
\end{eqnarray}
Thus, we may assume that \( f_{\varepsilon}(t) = f^*_{\varepsilon}(t) \) and \( L f_{\varepsilon}(t) = (L f_{\varepsilon})^*(t) \) for \( t \in (0,\infty) \); that is, \( f_{\varepsilon} \) is nonnegative and nonincreasing, and \( L f_{\varepsilon} \) is positive and decreasing.

Define 
\[
g_{a, \varepsilon}(t) := a^{1/p} f_{\varepsilon}(at).
\]
A straightforward calculation shows that
\[
L(g_{a,\varepsilon})(x) = a^{1/p-1} L(f_{\varepsilon})(x/a), \quad \text{and} \quad \|g_{a,\varepsilon}\|_{p,q} = \|f_{\varepsilon}\|_{p,q} = 1.
\]
Moreover, 
\[
\|L f_{\varepsilon}\|_{p',r} = \|L g_{a,\varepsilon}\|_{p',r} \geq c_1 - \varepsilon.
\]

For any given \( \delta > 0 \), there exist numbers \( b^-_{\varepsilon,\delta}, b^+_{\varepsilon,\delta}, d^-_{\varepsilon,\delta}, d^+_{\varepsilon,\delta} \) with 
\[
0 < b^-_{\varepsilon,\delta} < b^+_{\varepsilon,\delta} < \infty \quad \text{and} \quad 0 < d^-_{\varepsilon,\delta} < d^+_{\varepsilon,\delta} < \infty,
\]
such that
\begin{equation}\label{f_ep}
	\|f_{\varepsilon} \chi_{(0, b^-_{\varepsilon, \delta}) \cup (b^+_{\varepsilon,\delta}, \infty)}\|_{p,q} < \delta,
\end{equation}
\begin{equation}\label{Lf_ep}
	\|L f_{\varepsilon} \chi_{(0, d^-_{\varepsilon, \delta}) \cup (d^+_{\varepsilon,\delta}, \infty)}\|_{p',r} < \delta\, c_1.
\end{equation}

It follows that
\[
1 - \delta < \|f_{\varepsilon} \chi_{(b^-_{\varepsilon, \delta}, b^+_{\varepsilon,\delta})}\|_{p,q} \leq 1,
\]
and
\[
(1 - \delta)(c_1 - \varepsilon) < \|L f_{\varepsilon} \chi_{(d^-_{\varepsilon, \delta}, d^+_{\varepsilon,\delta})}\|_{p',r} \leq c_1.
\]

By the monotonicity of the estimates (i.e., as \( b^-_{\varepsilon, \delta} \) decreases to \( 0 \), inequality \eqref{f_ep} holds, and similarly for \( d^-_{\varepsilon, \delta} \)), we may assume that
\begin{equation}\label{Rerang 1}
	(1 + \delta) \int_{b^-_{\varepsilon, \delta}}^{b^+_{\varepsilon,\delta}} \left( t^{1/p - 1/q} \bigl(f_{\varepsilon} \chi_{(b^-_{\varepsilon, \delta}, b^+_{\varepsilon,\delta})}\bigr)^*(t) \right)^q dt \geq \|f_{\varepsilon} \chi_{(b^-_{\varepsilon, \delta}, b^+_{\varepsilon,\delta})}\|_{p,q},
\end{equation}
\begin{equation}\label{Rerang 2}
	(1 + \delta) \int_{d^-_{\varepsilon, \delta}}^{d^+_{\varepsilon,\delta}} \left( t^{1/p' - 1/q} \bigl(L f_{\varepsilon} \chi_{(d^-_{\varepsilon, \delta}, d^+_{\varepsilon,\delta})}\bigr)^*(t) \right)^r dt \geq \|L f_{\varepsilon} \chi_{(d^-_{\varepsilon, \delta}, d^+_{\varepsilon,\delta})}\|_{p',r}.
\end{equation}

For \( 0 < \delta_1 < \delta \), we can choose
\[
0 < b^-_{\varepsilon,\delta_1} < b^-_{\varepsilon,\delta} < b^+_{\varepsilon,\delta} < b^+_{\varepsilon,\delta_1} < \infty
\]
and
\[
0 < d^-_{\varepsilon,\delta_1} < d^-_{\varepsilon,\delta} < d^+_{\varepsilon,\delta} < d^+_{\varepsilon,\delta_1} < \infty.
\]

Replacing \( f_{\varepsilon} \) with \( g_{a,\varepsilon} \) in \eqref{f_ep} and \eqref{Lf_ep}, a straightforward computation shows that
\begin{equation}\label{f_ep2}
	\|g_{a,\varepsilon} \chi_{(0, b^-_{\varepsilon, \delta}/a) \cup (b^+_{\varepsilon,\delta}/a, \infty)}\|_{p,q} < \delta,
\end{equation}
\begin{equation}\label{Lf_ep2}
	\|L g_{a,\varepsilon} \chi_{(0, a d^-_{\varepsilon, \delta}) \cup (a d^+_{\varepsilon,\delta}, \infty)}\|_{p',r} < \delta\, c_1.
\end{equation}

Now, construct an increasing sequence \( \{a_i\} \) with \( 1 < a_i < a_{i+1} \) and \( a_i \to \infty \) such that
\[
0 < \frac{b^+_{\varepsilon, \varepsilon 2^{-i}}}{a_i} \leq \frac{b^-_{\varepsilon, \varepsilon 2^{-i+1}}}{a_{i-1}} < d^+_{\varepsilon, \varepsilon 2^{-i+1}} a_{i-1} \leq d^-_{\varepsilon, \varepsilon 2^{-i}} a_i,
\]
for all \( i \in \mathbb{N} \).

Define the disjoint intervals
\[
I_i := \left(\frac{b^-_{\varepsilon, \varepsilon 2^{-i}}}{a_i}, \frac{b^+_{\varepsilon, \varepsilon 2^{-i}}}{a_i} \right] \quad \text{and} \quad J_i := \left[ d^-_{\varepsilon, \varepsilon 2^{-i}} a_i, d^+_{\varepsilon, \varepsilon 2^{-i}} a_i \right).
\]
Then we have
\[
\|L g_{a_i, \varepsilon} \chi_{(0,\infty) \setminus J_i}\|_{p',r} < c_1\, \varepsilon 2^{-i} \quad \text{and} \quad \|g_{a_i,\varepsilon} \chi_{(0,\infty) \setminus I_i}\|_{p,q} < \varepsilon 2^{-i}.
\]
Moreover, we may assume that
\[
\frac{b^+_{\varepsilon, \varepsilon 2^{-i}}}{a_i} = \frac{b^-_{\varepsilon, \varepsilon 2^{-i+1}}}{a_{i-1}} \quad \text{and} \quad d^+_{\varepsilon, \varepsilon 2^{-i+1}} a_{i-1} = d^-_{\varepsilon, \varepsilon 2^{-i}} a_i.
\]
Set
\[
I_1 = \left(\frac{b^-_{\varepsilon, \varepsilon 2^{-1}}}{a_1}, \infty \right) \quad \text{and} \quad J_1 = \left(0, d^+_{\varepsilon, \varepsilon 2^{-1}} a_1 \right).
\]
It follows that
\[
(0, \infty) = \bigcup_i I_i = \bigcup_i J_i.
\]

Next, define the functions
\[
\bar{g}_{a_i,\varepsilon} := g_{a_i,\varepsilon} \chi_{I_i} \quad \text{and} \quad \bar{h}_{a_i, \varepsilon} := L g_{a_i,\varepsilon} \chi_{J_i}.
\]
Clearly, the families \( \{\bar{g}_{a_i,\varepsilon}\}_i \) and \( \{\bar{h}_{a_i,\varepsilon}\}_i \) consist of functions with pairwise disjoint supports. Moreover, we have
\begin{equation}\label{bar g_i - g_i}
	\|\bar{g}_{a_i,\varepsilon} - g_{a_i,\varepsilon}\|_{p,q} < \varepsilon 2^{-i} \quad \text{and} \quad \|\bar{h}_{a_i,\varepsilon} - L g_{a_i,\varepsilon}\|_{p',r} < c_1\, \varepsilon 2^{-i}.
\end{equation}
This implies
\[
\|g_{a_i,\varepsilon} \chi_{(0,\infty) \setminus I_i}\|_{p,q} < \varepsilon 2^{-i} \quad \text{and} \quad \|L g_{a_i,\varepsilon} \chi_{(0,\infty) \setminus J_i}\|_{p',r} < c_1\, \varepsilon 2^{-i}.
\]
Finally, note that the conditions \eqref{Rerang 1} and \eqref{Rerang 2} remain valid with \( \delta \) replaced by \( \varepsilon 2^{-i} \).

\begin{lemma} \label{Lemma 6}
	For the set of functions \( \{g_{a_i,\varepsilon}\}_{i=1}^{\infty} \) and any sequence \( \{\alpha_i\} \in \ell_q \), the following inequalities hold:
	\[
	\left(\frac{1}{(1+\varepsilon)^2} - 4\varepsilon \right) \sum_{i=1}^{\infty} |\alpha_i|^q \leq \left\| \sum_{i=1}^{\infty} \alpha_i g_{a_i,\varepsilon} \right\|_{p,q}^q,
	\]
	\[
	\left(\frac{1}{(1+\varepsilon)^2} - 4\varepsilon \right) c_1 \sum_{i=1}^{\infty} |\alpha_i|^q \leq \left\| \sum_{i=1}^{\infty} \alpha_i L g_{a_i,\varepsilon} \right\|_{p',r}^q.
	\]
\end{lemma}

\begin{proof}
	Using inequality \eqref{bar g_i - g_i} together with Lemma~\ref{Lemma 3.4}, we obtain
	\begin{equation}\label{lower 1}
		\left\| \sum_{i=1}^{\infty} \alpha_i \bar{g}_{a_i,\varepsilon} \right\|_{p,q} - 2\varepsilon\, \sup_i |\alpha_i| \leq \left\| \sum_{i=1}^{\infty} \alpha_i g_{a_i,\varepsilon} \right\|_{p,q}.
	\end{equation}
	By the disjointness of the supports of \( \bar{g}_{a_i,\varepsilon} \) (see \eqref{SystDisj}) and applying Lemma~\ref{SumNonIncrRear}, we obtain
	\begin{align*}
		\left\| \sum_{i=1}^{\infty} \alpha_i \bar{g}_{a_i,\varepsilon} \right\|_{p,q}^q 
		&= \int_{0}^{\infty} \Bigl( t^{1/p - 1/q} \Bigl(\sum_{i=1}^{\infty} \alpha_i \bar{g}_{a_i,\varepsilon}\Bigr)^*(t) \Bigr)^q dt\\
		&\ge \sum_{i=1}^{\infty} |\alpha_i|^q \int_{I_i} \Bigl( t^{1/p - 1/q} \bar{g}_{a_i,\varepsilon}^*(t) \Bigr)^q dt\\
		&\ge \frac{1}{1+\varepsilon} \sum_{i=1}^{\infty} |\alpha_i|^q \|\bar{g}_{a_i,\varepsilon}\|_{p,q}^q\\
		&\ge \frac{1}{(1+\varepsilon)^2} \sum_{i=1}^{\infty} |\alpha_i|^q \|\bar{g}_{a_i,\varepsilon}\|_{p,q}^q - 2\varepsilon \sup_i |\alpha_i|.
	\end{align*}
	Thus, using  $\|g_{a_i,\varepsilon}\|_{p,q} = 1$ and  combining the above inequality with \eqref{lower 1}, we conclude
	\[
	\left\| \sum_{i=1}^{\infty} \alpha_i g_{a_i,\varepsilon} \right\|_{p,q}^q \ge \left(\frac{1}{(1+\varepsilon)^2} - 4\varepsilon\right) \sum_{i=1}^{\infty} |\alpha_i|^q.
	\]
	The second inequality, for \( \{L g_{a_i,\varepsilon}\} \), is obtained by a similar argument.
\end{proof}

\begin{lemma} \label{Lemma 7}
	For the set of functions \( \{g_{a_i,\varepsilon}\}_{i=1}^{\infty} \) and any sequence \( \{\alpha_i\} \in \ell_q \), the following inequalities hold:
	\[
	\left\| \sum_{i=1}^{\infty} \alpha_i g_{a_i,\varepsilon} \right\|_{p,q}^q \le \left(\frac{1}{(1+\varepsilon)^2} + 4\varepsilon \right) \sum_{i=1}^{\infty} |\alpha_i|^q,
	\]
	\[
	\left\| \sum_{i=1}^{\infty} \alpha_i L g_{a_i,\varepsilon} \right\|_{p',r}^q \le \left(\frac{1}{(1+\varepsilon)^2} + 4\varepsilon \right) c_1 \sum_{i=1}^{\infty} |\alpha_i|^q.
	\]
\end{lemma}

\begin{proof}
	Using the inequality
	\[
	\left\| \sum_{i=1}^{\infty} \alpha_i g_{a_i,\varepsilon} \right\|_{p,q} \le \left\| \sum_{i=1}^{\infty} \alpha_i \bar{g}_{a_i,\varepsilon} \right\|_{p,q} + 2\varepsilon\, \sup_i |\alpha_i|,
	\]
	and following an argument similar to that in the proof of Lemma~\ref{Lemma 6}, one obtains the desired estimates. (The details are similar to those in the previous lemma and are omitted.)
\end{proof}

\bigskip

\noindent\textbf{Proof of Theorems \ref{Thereom 1} and \ref{Thereom 2}.}

By combining Lemmas \ref{Lemma 6} and \ref{Lemma 7}, we demonstrate that there exists an infinite-dimensional subspace of \( L^p(0,\infty) \) on which the Laplace transform \( L \) is invertible. Specifically, the operator \( L \) maps 
\[
\operatorname{span}\Bigl(\{g_{a_i,\varepsilon}\}_{i=1}^\infty\Bigr)
\]
onto 
\[
\operatorname{span}\Bigl(\{L g_{a_i,\varepsilon}\}_{i=1}^\infty\Bigr).
\]
It follows that the Bernstein numbers satisfy
\[
b_n(L) = \|L\| = c_1.
\]
Thus, \( L \) is not strictly singular, which establishes Theorem \ref{Thereom 2}.

Finally, by restricting \( L \) to \( \operatorname{span}\Bigl(\{L g_{a_i,\varepsilon}\}_{i=1}^\infty\Bigr) \) and applying observation I., we conclude that 
\[
\beta(L) = c_1 = \|L\|,
\]
which gives Theorem \ref{Thereom 1}.

\bigskip

\textbf{Note:} It is natural to ask whether the Laplace transform
\[
L : L^{p,q}(0,\infty) \to L^{p',r}(0,\infty),
\]
for \( q < r \), is strictly singular or even finitely strictly singular.

\bibliographystyle{amsalpha} 
\bibliography{reference03} 

\end{document}